\documentclass[10pt]{amsart}
\usepackage{amssymb,latexsym,exscale}

\newcommand{\mf}[1]{\ensuremath{\mathfrak{#1}}}

\newcommand{\G}{\mathcal G}
\newcommand{\calL}{\mathcal L}

\newcommand{\M}{\mf{M}}

\newcommand{\calO}{\mathcal O}

\newcommand{\calS}{\mathcal S}

\def\ran{\mathop{\rm ran}\nolimits}
\def\Min{\mathop{\rm Min}\nolimits}

\newcommand{\ba}{\bar{a}}

\newcommand{\be}{\bar{e}}

\newcommand{\RM}{\mathrm {RM}}
\newcommand{\acl}{\mathrm {acl}}
\newcommand{\dcl}{\mathrm {dcl}}

\newcommand{\eq}{\mathrm {eq}}

\newcommand{\red}{\mathrm {red}}
\newcommand{\RV}{\mathrm {RV}}
\newcommand{\Th}{\mathrm {Th}}
\newcommand{\tp}{\mathrm {tp}}

\newcommand{\Div}{\mathrm {div}}
\newcommand{\con}{\mathrm {con}}
\newcommand{\an}{\mathrm {an}}

\newcommand{\im}{\mathrm{Im}}

\def\lceil{\ulcorner}
\def\rceil{\urcorner}

\newcommand{\Uu}{\mathcal U}
\newcommand{\Oo}{\mathcal O}

\begin{document}

\title[Lack of EI with analytic functions]{Unexpected imaginaries in valued fields with analytic structure}
\author{Deirdre Haskell}
\address{Department of Mathematics and Statistics\\
	  McMaster University \\
	  1280 Main St W. \\
         Hamilton ON L8S 4K1, Canada} 
\email{haskell@math.mcmaster.ca}
\author{Ehud Hrushovski}
\address{Institute of Mathematics\\
	Hebrew University\\
	Jerusalem 91904, Israel}
\email{ehud@math.huji.ac.il}
\author{Dugald Macpherson}
\address{School of Mathematics\\
	University of Leeds\\
	 Leeds LS2 9JT, UK}
\email{pmthdm@maths.leeds.ac.uk}

\newtheorem{defi}{Definition}[section]
\newtheorem{theorem}[defi]{Theorem}
\newtheorem{definition}[defi]{Definition}
\newtheorem{lemma}[defi]{Lemma}
\newtheorem{proposition}[defi]{Proposition}
\newtheorem{conjecture}[defi]{Conjecture}
\newtheorem{corollary}[defi]{Corollary}
\newtheorem{problem}[defi]{Problem}
\newtheorem{remark}[defi]{Remark}
\newtheorem{example}[defi]{Example}
\newtheorem{question}[defi]{Question}
\newtheorem{convention}[defi]{Conventions}
\newtheorem{fact}[defi]{Fact}
\newtheorem{notation}[defi]{Notation}


\thanks{This work arose from extensive efforts to prove elimination of imaginaries in the algebraically closed 
and real closed cases. We express our appreciation for the efforts of David Lippel and Tim Mellor, whose work contributed 
to our understanding of the problems involved. We also thank Leonard Lipshitz for useful comments. The authors acknowledge the financial support of NSERC, the Israel Science Foundation (1048/07) and the EPSRC (EP/F068751/1) respectively.} 

\date{\today}

\def\Ind#1#2{#1\setbox0=\hbox{$#1x$}\kern\wd0\hbox to 0pt{\hss$#1\mid$\hss}
\lower.9\ht0\hbox to 0pt{\hss$#1\smile$\hss}\kern\wd0}
\def\dnf{\mathop{\mathpalette\Ind{}}}
\def\Notind#1#2{#1\setbox0=\hbox{$#1x$}\kern\wd0\hbox to 0pt{\mathchardef
\nn=12854\hss$#1\nn$\kern1.4\wd0\hss}\hbox to
0pt{\hss$#1\mid$\hss}\lower.9\ht0 \hbox to
0pt{\hss$#1\smile$\hss}\kern\wd0}
\def\df{\mathop{\mathpalette\Notind{}}}

\begin{abstract}
We give an example of an imaginary defined in certain valued fields with  analytic structure which cannot
be coded in the `geometric' sorts which suffice to code all imaginaries in the corresponding algebraic setting.  
\end{abstract}
\maketitle

\section{Introduction}

The  work of \cite{DvdD} on quantifier elimination for valued fields with  analytic functions illustrated
the power of Weierstrass preparation and gave rise to the intuition that restricted analytic functions do not significantly
increase the collection of definable sets beyond those which are already given by the algebraic structure. This is
certainly true for sets in one variable as, by the work of \cite{LR} in the algebraically closed case, of \cite{DHM} in
the $p$-adically closed case and of \cite{DL} in the real closed case, the theory of the valued field with restricted analytic 
functions in the appropriate language is respectively $C$-minimal, $P$-minimal or weakly o-minimal. This intuition gave rise
to a belief that the theory of a valued field with restricted analytic functions should  eliminate imaginaries to the same
`geometric sorts' which suffice to eliminate imaginaries in the  algebraic situation. In this paper, we show that this belief
is false. We give an example of an imaginary which arises in the analytic setting and which cannot be coded in the 
geometric sorts. The example has versions in each of the above three settings.

Let $K$ be a field, and $v:K\to\Gamma\cup\{\infty\}$ a valuation map. 
  Let $\calL=(+,-,.,0,1,\Div)$ be the 
language of valued rings, where $\Div$ is a binary relation symbol interpreted on $K$ by putting $\Div(x,y)$ whenever $v(x)\leq v(y)$. 
This is a one-sorted language. For an arbitrary valued field $F$, we shall write $\Gamma(F)$, $\calO(F)$, $\mf{M}(F)$, and $k(F)$, 
for, respectively, the value group, valuation ring, maximal ideal, and residue field of $F$.

We also consider the multi-sorted language $\calL_{\G}$ in which ACVF (the theory of algebraically closed valued fields) was proved in 
\cite{HHM} to have elimination of imaginaries. This has sorts $\Gamma$ (for the value group) and $k$ (for the residue field), and also, 
for each $n>0$, a sort $S_n$ and  a sort $T_n$. We view these as sorts for arbitrary valued fields --  $\calL_{\G}$ is a general 
sorted language for valued fields. The members of $S_n$ are codes for rank $n$ lattices for 
the valuation ring; that is, given a valued field $F$ with valuation ring $\calO=\calO(F)$, the members of $S_n(F)$ are codes for free rank 
$n$  $\calO$-submodules of $F^n$. Equivalently, as ${\rm GL}_n(F)$ acts transitively on the space of such lattices and the stabiliser of 
the lattice $\calO^n$ is ${\rm GL}_n(\calO)$, we may view the elements of $S_n(F)$ as codes for left cosets of ${\rm GL}_n(\calO)$ 
in ${\rm GL}_n(F)$.
If $s\in S_n(F)$ codes the lattice $\Lambda$, then $\Lambda/\mf{M}\Lambda$ has the structure of an $n$-dimensional vector space 
over $k(F)$, and the set $T_n$ consists of pairs $(s,t)$ where $s\in S_n(F)$ codes some $\Lambda$ as above, and $t$ codes an 
element of $\Lambda/\mf{M}\Lambda$. Thus there is a canonical surjection $\pi_n:T_n(F)\to S_n(F)$ such that each fibre is a set of 
codes for elements of an $n$-dimensional $k(F)$-space.
The sort $\Gamma$ is redundant (but included to follow the conventions of \cite{HHM}), since it is naturally identified with $S_1$, 
identifying $\gamma\in \Gamma$ with the lattice $\gamma\calO(F):=\{x\in F:v(x)\geq \gamma\}$. Likewise, the sort $k$ is 
redundant. 

We also often refer to the sort
$RV$. For the valued field $F$, $RV(F)$ consists of codes for elements of
 the set $F^*/1+\mf{M}(F)$. However, $RV$  is not formally a sort for $\calL_{\G}$. One can identify $RV(F)$ with an 
 $\emptyset$-definable subset of $T_1(F)$: the element  $b(1+\mf{M})$ in $RV$ is identified with the open ball $b+b\mf{M}$ 
 which is an element of $b\calO/b\mf{M}$.
 There is a natural map $RV\to \Gamma$, also denoted by $v$.

We first describe the context of algebraically closed valued fields.
Let $K_0$ be a complete algebraically closed valued field of characteristic zero, with value group $\Gamma_0$, value map $v:K_0\to \Gamma_0$, valuation ring $\calO_0$, maximal ideal
$\mf{M}_0$, residue field $k_0$.
We consider the rings of separated power series over $K_0$, as introduced by  Lipshitz in \cite{L}. We shall not here describe the full setting, 
but for any $n,m\geq 0$ there is a ring $S_{m,n}$ of power series  in variables $x=(x_1,\ldots,x_m)$ ranging through $\calO_0^m$ and 
$\rho=(\rho_1,\ldots,\rho_n)$ ranging through $\mf{M}_0^n$. In the language $\calL^{\an}_D$ there is a function symbol for each element 
of this ring. In the standard model $K_0$ these function symbols are interpreted by the corresponding power series functions on 
$\calO_0^m \times \mf{M}_0^n$ (where they converge), and they  take value 0 on any $(a,b)$ where $a=(a_1,\ldots,a_m)$, 
$b=(b_1,\ldots,b_n)$, and
$(a,b)\not\in \calO_0^m\times \mf{M}_0^n$. The language also has two binary functions symbols $D_0,D_1$ for truncated division with range $\calO_0$ and $\mf{M}_0$ respectively. Let 
$T^{\an}_D$ be the theory of the valued field in the language $\calL^{\an}_D$, with the above function symbols interpreted in the natural 
way. Lipshitz proved in \cite[Theorem 3.8.2]{L} that this theory, parsed in a three-sorted language with sorts for the valuation ring, the maximal ideal and the value group, has quantifier elimination. A version with two sorts (the field and the value group) is stated in \cite{CL} (Theorem 4.5.15).

With $F$  a model of $T^{\an}_D$, we may extend the language $\calL^{\an}_D$ to
a language $\calL^{\an}_{D,\G}$ by adjoining the sorts $k,\Gamma,S_n,T_n$ (for $n\geq 1$)
 from $F^{\eq}$, and expand $F$ correspondingly. Let  $T^{\an}_{D,\G}$ be the resulting multi-sorted theory. We prove

\begin{theorem} \label{theorem1}
The theory $T^{\an}_{D,\G}$ does not have elimination of imaginaries.
\end{theorem}

The basic idea of the proof is to consider the graph of a restriction of the exponential function, which is definable in $\calL^{\an}$. The domain and range are definable in $\calL$ and so coded in $\calL_G$, but the exponential map is not definable in $\calL_\G$. Its graph is a group, and we show that a generic torsor of this group is not coded in $\G$.

We next consider real closed valued fields with analytic structure. A {\em real closed valued field} $F$ is a real closed field equipped with a valuation arising from a proper non-trivial convex valuation ring. It may be viewed in the language $\calL_{<}=\calL\cup\{<\}$, where $<$ is a binary relation symbol interpreted by the ordering. The theory of real closed valued fields is complete, and has quantifier elimination, by
\cite{CD}. Furthermore, by \cite{Di}, its theory  is {\em weakly o-minimal}: in all models of the theory, every definable subset of the field is a finite union of convex sets, not necessarily with endpoints in $F\cup\{\pm\infty\}$. If $\calL_{<,\G}$ is the multi-sorted language with sorts $\G$ as above, then the theory of the expansion of $F$ to $\calL_{<,\G}$ has elimination of imaginaries, by \cite{M}.

Now consider 
any  o-minimal expansion $\bar{{\mathbb R}}$ of the real  field ${\mathbb R}$, in a language $\bar{\calL}$,
 and let $T=\Th(\bar{{\mathbb R}})$. 
Assume that exponentiation restricted to $[0,1]$ is definable in $\bar{{\mathbb R}}$.
Assume also that $T$ is {\em polynomially bounded}, that is, for every definable partial function 
${\mathbb R}\to {\mathbb R}$, there are $a\in {\mathbb R}$ and $d\in {\mathbb N}$ such that 
$f(x)\leq x^d$ for all $x>a$.   By \cite{MILLER},
 this is equivalent to assuming that unrestricted exponentiation is not definable in 
$\bar{{\mathbb R}}$.  A familiar example of 
such a structure $\bar{{\mathbb R}}$ is ${\mathbb R}_{{\rm an}}$, the expansion 
of the real field by all restricted analytic functions -- see for example \cite{DvdD}. Let $\bar{F}\models T$ be non-archimedean.
Let 
$$V:=\{x\in F: \mbox{~for some~} n\in {\mathbb N}^{>0} (|x|<n)\},
$$
 the subring of $\bar{F}$ consisting of finite elements. Then $V$ is a convex valuation ring of $\bar{F}$ with maximal ideal 
 $\mf{M}$, the ideal of infinitesimals. 
Observe that $f(V)\subseteq V$ for every continuous $\emptyset$-definable function $f:V\to {\mathbb R}$.
Following \cite{DL}, let $T_{\con}=\Th((\bar{F},V))$ in a language $\bar{\calL}_{\con}$ obtained from $\bar{\calL}$ by adjoining a unary 
predicate $P$ interpreted by $V$.

It is well-known (see e.g. \cite[(1.2) p. 94]{Dries}) that any o-minimal expansion of an ordered field has elimination of imaginaries. Likewise, as noted above, RCVF has elimination of imaginaries in the multi-sorted language
$\calL_{<,\G}$, and we had hoped that these two results could be combined.
However, let $\bar{\calL}_{\con,\G}$ be the extension of $\bar{\calL}_{\con}$ obtained by adding the sorts $S_n$ and $T_n$ (as well as $k$, $\Gamma$) of $\calL_{<,\G}$. 
Let $T_{\con,\G}$ be the theory of the natural expansion of $(\bar{F},V)\models T_{\con}$ to the language
$\bar{\calL}_{\con,\G}$. We prove

\begin{theorem} \label{theorem2}

The  theory $T_{\con,\G}$ does not have elimination of imaginaries.
\end{theorem}

Finally, we consider the $p$-adic setting. Here, the structure of interest is ${\mathbb Q}_p^{\an}$ in the language $\calL^{\an}_{p,D}$, in the context of  \cite{DvdD}. The $p$-adic field ${\mathbb Q}_p$, equipped with the $p$-adic valuation $v_p$,  is considered in 
Macintyre's language $\calL_p$ with unary predicates $P_n$ (for $n>1$), where $P_n$ is interpreted by the set of $n^{th}$ powers in 
${\mathbb Q}_p$. The language $\calL^{\an}_{p,D}$ has a binary predicate $D$ for (truncated) division, and a function symbol for each 
restricted analytic function
${\mathbb Q}_p^m \to {\mathbb Q}_p$ (for $m\geq 0$) which is defined by a convergent power series on ${\mathbb Z}_p^m$ (coefficients tending in valuation to infinity), and takes the  
value 0 off ${\mathbb Z}_p^m$. 

Let $\calL_{p,\G}$ be the extension of $\calL_p$ with sorts $S_n$ for each $n>0$; the $T_n$ are not needed since elements of the $T_n$ 
are coded in the other sorts, as the value group is discrete. By \cite{HM}, the theory of ${\mathbb Q}_p$ has elimination of imaginaries in 
the (semi-algebraic) language $\calL_{p,\G}$. This result was used in \cite{HM} to prove rationality results for certain Poincar\'e series for 
groups. The $p$-adic analytic quantifier elimination was used to prove rationality of Poincar\'e series for compact $p$-adic analytic groups 
by du Sautoy in \cite{S}. The hope was to extend these results using analytic elimination of imaginaries. However, we have the following
theorem.  Let $\calL_{p,\G}^{\an}$ be the corresponding extension of $\calL_{p,\G}$ by the restricted analytic functions. We let 
$T^{\an}_{p,\G}$ be the theory of ${\mathbb Q}_p^{\an}$ in the subanalytic language
$\calL_{p,\G}^{\an}$. 

\begin{theorem}\label{theorem3}
The theory $T_{p,\G}^{\an}$ does not have elimination of imaginaries.
\end{theorem}

We give the proof of Theorem~\ref{theorem1} in Section 2. We have included proofs of several intermediate lemmas, possibly of wider interest, which may be known but are hard to find in the literature. The proofs of Theorems~\ref{theorem2} and \ref{theorem3} are similar, and are given in Sections 3 and 4 respectively. Section 5 contains a sketch of an alternative proof of Theorem~\ref{theorem1} in residue characteristic 0 which is perhaps slightly  shorter (though it rests on some of the lemmas in Section 2) but may not generalise so well to the o-minimal and $p$-adic contexts.

\begin{remark}\rm
As mentioned above and referenced later, the theories in Theorems~\ref{theorem1}, \ref{theorem2} and \ref{theorem3} are, respectively, $C$-minimal, weakly o-minimal, and $P$-minimal. It follows immediately that these theories eliminate the quantifier $\exists^\infty$ (in the field sort); that is, for any uniformly definable family of subsets of the field, there is a finite upper bound on the sizes of the finite members. This is used without explicit mention.
\end{remark}

It may be that in these settings, elimination of imaginaries can be obtained by adding a few further clearly described sorts. Let us say that an imaginary set $Y$ is {\em coded} over $A$ in sorts
 $\calS$ if there is an $A$-definable embedding from $Y$ into a product of the sorts in
 $\calS$. In this paper we have shown that in the above three settings, a certain imaginary is not coded 
in the geometric sorts. This is seen most cleanly in Section 5, where is is shown that for $ACVF^{\an}$ in residue characteristic 0, if $E$ is the graph of exponentiation and $G=G_a \times G_m$, then $G(\calO)/E$ is not coded in the geometric sorts.

\begin{question}\rm Is it true that, in the three settings considered in this paper, any imaginary can be coded in one of the form ${\rm GL}_n/H$ for some definable group $H$?
\end{question}

\noindent
{\em Conventions and Notation.}

 In each of the three settings, we use the symbol $\Uu$ for a large saturated model
(without specifying the sorts of $\Uu$). Though we use $VF$ for the home sort, the underlying field of $\Uu$ is denoted by $K$. The valuation 
ring of $K$ is denoted by $\Oo$, its maximal ideal is denoted by $\mf{M}$, the residue field by $k$, and the value group by $\Gamma$. There 
is a potential confusion between viewing $\Gamma$ (or other sort symbols) as a symbol for a sort, or as the value group of $K$, but since 
we never move outside $\Uu$ this should be unproblematic. We often write $x\in VF$, meaning that $x\in VF(\Uu)$, and treat the other sorts similarly. If $M$ is an elementary submodel of $\Uu$, we write $VF(M)$ for the home
 sort of $M$, $\Oo(M)$ for the valuation ring of $M$, and so on.
If $\gamma\in \Gamma$ we put $\gamma\Oo:=\{x\in \Uu: v(x)\geq \gamma\}$ and $\gamma\mf{M}:=\{x\in \Uu:v(x)>\gamma\}$.   Throughout the paper, if $A$ is a subset of a model $M$, 
then $\acl(A)$ denotes the model-theoretic algebraic closure of $A$ in $M$, rather than the field-theoretic algebraic closure. 

Recall that if $G$ is a group, then a {\em torsor~} or {\em principal homogeneous space} for
 $G$ is a set $X$ equipped with a regular (that is, sharply 1-transitive) action of $G$
 on $X$. Let $X,Y$ be torsors of the groups $G,H$ respectively, with the actions of $G$
 on $X$ and $H$ on $Y$ both denoted by $*$. Then an {\em affine homomorphism} $X \to Y$
is a pair $(f,c)$ where $c:X \to Y$ is a function, $f:G \to H$ is a group homomorphism (the {\em homogeneous component} of 
$(f,c)$), and $c(g*x))=f(g)*c(x)$ for all $g\in G$ and $x\in X$. If $(f_1,c)$ and $(f_2,c)$ are both affine homomorphisms $X\to Y$
 then $f_1=f_2$, so we sometimes denote the affine homomorphism $(f,c)$  just by $c$. The imaginary that we exhibit in each
 case is an affine homomorphism whose homogeneous component is essentially exponentiation.
 
 If $v:F \to \Gamma \cup\{\infty\}$ is a valuation on a field $F$, then, for $a\in F$ and $\gamma \in \Gamma\cup\{\infty\}$, an {\em open ball} is a set of the form $B_{>\gamma}(a):=\{x\in F: v(x-a)>\gamma\}$, and a {\em closed ball} has
 form
$B_{\geq \gamma}(a):=\{x: v(x-a)\geq \gamma\}$ (so may be a singleton). We write $B_\gamma(a)$
if we do not wish to specify whether the ball is open or closed.
 
 We view definable sets such as balls both as imaginaries and as sets of (field) elements, viewed in the monster model. For example, a ball $B$ may be viewed as $\{x\in VF:x\in B\}$. When viewed as an imaginary, we often denote it as $\lceil B\rceil$. 
 
 We shall make heavy use of $C$-minimality, (weak) o-minimality, and $P$-minimality, often without detailed explanation. We assume, for example, that the reader can picture what kinds of sets are definable in a $C$-minimal expansion of a valued field. (Formally, they are Boolean combination of balls, and can be described -- canonically -- as finite unions of `Swiss cheeses' in the language of Holly \cite[Theorem 3.26]{holly}.)
 
 If $B$ is a closed ball in some valued field, then the {\em reduction} $\red(B)$ of $B$ is the collection of open sub-balls of $B$ of the same radius. This is a set in parameter-definable bijection with the residue field $k$.
This notation is occasionally extended. If, for example, $B$ has radius $\delta$, $\beta>\delta$, and $W$ is the collection of closed (or open) sub-balls of $B$ of radius $\beta$, then we may view $W$ as a closed ball of radius $\delta$ whose {\em elements} are closed balls of radius $\beta$, so $W$ is a {\em 1-torsor} in the language of \cite[Section 2.3]{HHM}. We may define $\red(W)$ as above, and $\red(W)$ is in $(\lceil B\rceil,\beta)$-definable bijection with $\red(B)$.

\section{The algebraically closed case -- proof of Theorem~\ref{theorem1}.}

As in the introduction, we assume that $K_0$ is a complete algebraically closed valued field of characteristic 0, with $K_0\models T^{\an}_D$, and let $T^{\an}_{D,\G}$ be the extension of $T^{\an}_D$ to the multi-sorted language
$\calL^{\an}_{D,\G}$. 

Recall the notion of a $C$-minimal theory, introduced in \cite{MS} and developed in \cite{HM1}.
A slightly more restricted notion, which also fits the present context, was developed further in \cite{HK}. The complete theory of an expansion of a valued field is {\em $C$-minimal} if, in any model $F$, any (parameter)-definable subset of the field is a Boolean combination of open or closed balls. 

\begin{theorem} \label{QEC-min}
\cite{LR} The theory $T^{\an}_D$  is $C$-minimal.
\end{theorem}

The next lemma is presumably well-known, and yields that algebraic closure defines a pregeometry on $K$.

\begin{lemma}\label{exchange}
 Algebraic closure has the exchange property in every model of $T^{\an}_D$.
\end{lemma}

\begin{proof} Suppose that this is false. Then by $C$-minimality and \cite[Proposition 6.1]{HM1}, there are definable infinite subsets $U,V$ of $K\models T^{\an}_D$ and a definable surjection $f:U \to V$, such that $f^{-1}(v)$ is an infinite ball  for all $v\in V$. Let $X:=\{(u,f(u)):u\in U\}\subset K^2$. Then it follows from \cite[Theorem 6.6]{celikler} that $X$ has non-empty interior. This is clearly impossible. 
\end{proof}

For any parameter set $C$ and  tuple $\be$ from $VF$, we define $\dim(\be/C)$ to be the length of a minimal subtuple $\be'$ of $\be$ such that $\be\in \acl(C,\be')$.

\begin{lemma}\label{stablyembedded}
(i) The field $k(\Uu)$ is a strongly minimal set in $\Uu$, and 
the ordered abelian group $\Gamma(\Uu)$, equipped with the induced structure, is o-minimal.

(ii) The value group $\Gamma(\Uu)$ has the structure of a {\em pure} divisible ordered abelian group (expanded by constants for elements of a subgroup $({\mathbb Q},<,+)$) and is stably embedded.
\end{lemma}

\begin{proof} (i) By Theorem~\ref{QEC-min}, $T^{\an}_{D}$ is $C$-minimal. The claims follow immediately from $C$-minimality, and the fact that $k(\Uu)$ is infinite. 

(ii) The quantifier elimination for $T^{\an}_D$ given in \cite{L} takes place in a 3-sorted language, with sorts  for the valuation ring, the maximal 
ideal and the value group. Thus, it suffices to observe that if an atomic formula $\phi(x)$ defines a subset of $\Gamma^n$, then there is a 
quantifier-free formula $\psi(x)$ in the language of ordered abelian groups which defines the same set, and if $\phi$ is over $\emptyset$, 
so is $\psi$. This is immediate. 
\end{proof}

\begin{lemma} \label{finiteimage}
Let $\Uu\models T^{\an}_{D,\G}$, and let $f:X\to K$ be a definable map from a geometric sort other than $K$. Then $f$ has finite image.
\end{lemma}

\begin{proof} We give the proof where $X$ is the sort $S_n$ for some $n$. The argument for the $T_n$-sorts is similar, and as noted above, the 
sorts $\Gamma$ and $k$ are redundant. Let $Y\subset K^{n^2}$ be the set of free bases of $\calO$-lattices coded in  $S_n(K)$; that is, if 
$y=(y_1,...,y_n)$ where $y_i \in K^n$ for all $i$, then $y\in Y$ if and only if there is an $\calO$-lattice of the form $\calO y_1 \oplus \ldots \oplus \calO y_n$, coded in $S_n$. Thus, $Y$ is just the set of linearly independent $n$-tuples from $K^n$. There is a map $g: Y\to K$ given by 
$g(y)=f(\calO y_1\oplus\ldots\oplus\calO y_n)$. 

If $y=(y_1,\ldots,y_n)\in Y$, then it is easily checked that there is $\gamma \in \Gamma^{>0}$ 
so that the following holds: if $(w_1,\ldots,w_n)\in Y$ with $y_i=(y_{i1},\ldots,y_{in})$ and $w_i=(w_{i1},\ldots,w_{in})$ for each $i$, and $v(y_{ij}-w_{ij})>\gamma$ for each $i,j$, then
$\calO y_1\oplus \ldots \oplus \calO y_n=\calO w_1 \oplus \ldots \oplus \calO w_n$; that is, 
the map $g$ is  locally constant. Thus, 
the proof reduces to the following claim.

\medskip

{\em Claim.} For any $m>0$ and  any  $W\subseteq K^m$, any  locally constant definable $g:W\to K$ has finite image. 

{\em Proof of Claim.} We use induction on $m$. 
 The case $m=1$ follows immediately from Lemma~\ref{exchange}. 
For the inductive step, assume the result holds for $m'<m$, and that $m>1$, and let $\pi_1:K^m\to K^{m-1}$ be the projection to the first $m-1$ coordinates. 
For any $\bar{a}\in \pi_1(W)$ there is an induced locally constant partial map $g_{\ba}:K \to K$ given by $g_{\ba}(y)=g(\ba,y)$. By the $m=1$ case, $\im(g_{\ba})$ is a finite set, of bounded size (as $\bar{a}$ varies) by compactness, and we may suppose that $|{\rm Im}(g_{\ba})|=t$ for all $\ba\in \pi_1(W)$. 
For each $\ba\in\pi_1(W)$, let ${\rm Im}(g_{\ba})=\{l_1(\ba),\ldots,l_t(\ba)\}$. For each $i=1,\ldots,t$ there is $b_i$ with $g(\ba,b_i)=l_i(\ba)$, and there is an open neighbourhood $N_i$ of $(\ba,b_i)$ such that $g$ is constant on $N_i$. Hence, for any $\ba'\in \pi_1(N_i)$, $l_i(\ba)\in {\rm Im}(g_{\ba'})$. Put $N(\ba)=\pi_1(N_1)\cap\ldots \cap \pi_1(N_t)$. Then ${\rm Im}(g_{\ba'})={\rm Im}(g_{\ba})=\{l_1(\ba),\ldots,l_t(\ba)\}$ for all $\ba'\in N(\ba)$.
By elimination of finite  imaginaries for fields, $\im(g_{\ba})$ is coded by a finite tuple of field elements $h(\ba)$ of length $t'$, say. Write $h(\ba)=(h_1(\ba),\ldots,h_{t'}(\ba))$. Then each $h_i$ is constant on $N(\ba)$.  Thus, the $h_i$ are locally constant on $\pi_1(W)\subseteq K^{m-1}$, so by induction have finite image. Thus ${\rm Im}(g)$ is finite.
\end{proof}

\begin{proposition} \label{orthog}
The value group $\Gamma$ and the residue field $k$ are orthogonal, in the following sense:  if $\alpha\in k$
and $\gamma\in\Gamma$, then for any model $M$, 
$\Gamma(\acl(M\alpha)) = \Gamma(M)$ and $k(\acl(M\gamma))=k(M)$.
\end{proposition}

\begin{proof}
Suppose for a contradiction that $\delta \in \Gamma(\acl(M\alpha)) \setminus \Gamma(M)$, for some $\alpha \in k$. Then there is an $M$-definable function $f:k\to \Gamma$ with infinite range such that $f(\alpha)=\delta$. Since $k$ is strongly minimal, the fibres of $f$ are finite, and a total ordering is induced by $\im(f)\subset \Gamma$ on the set of fibres. This  contradicts strong minimality of $k$.

Similarly, suppose there is $\beta \in k(\acl(M\gamma))\setminus k(M)$, where $\gamma \in \Gamma$. Using coding of finite sets in the residue field, there is an $M$-definable function $g:\Gamma \to k$ with infinite range $X$. By definable choice in $\Gamma$, there is a definable injective function $h:X\to \Gamma$ such that $g(h(x))=x$ for all $x\in X$. Thus, the ordering on $\Gamma$ induces an ordering on $X$, which is  incompatible with the strong minimality of $k$. 
\end{proof}

Recall that, given a parameter set $C$,  an element $e\in K$ is said to be {\em generic} over $C$ in the $C$-definable ball $s$,
or in the  chain $s=\bigcap(s_i:i\in I)$ (ordered by inclusion) of $C$-definable balls, if $e$ lies in $s$ and there is no $\acl(C)$-definable proper sub-ball of $s$ containing $e$; this follows \cite[Definition 2.5.1]{HHM} and \cite[Definition 3.4]{HK}. By $C$-minimality (see e.g. \cite[Section 2.5]{HHM} or, expressed more generally,  \cite[Section 3]{HK}), under the assumptions of the following lemma,
for any such $s$ there is a unique ${\rm Aut}(\Uu/C)$-invariant type $p_s$ over $\Uu$
 such that: for any $C'$ with $C\subset C'\subset \Uu$, $p_s|C'$ is the type of
 elements generic in $s$ over $C'$. This notion of genericity extends to {\em 1-torsors} (where a ball is a collection of sub-balls of given radius, in the sense of Section 1).

\begin{lemma}\label{closedgeneric}
Let $C$ be a parameter set, $s$ a $C$-definable closed ball, and let $a\in K$ be a generic element of $s$ over $C$. Then
$\Gamma(C)=\Gamma(Ca)$.
\end{lemma}

\begin{proof} The proof of Lemma 2.5.5 $(i) \Rightarrow (ii)$ of \cite{HHM} just uses $C$-minimality in the sense of \cite{HK}, so goes through to the current context. See also \cite[Lemma 3.19]{HK}.
\end{proof}

\begin{lemma}\label{weightone}
 Let $C=\acl(C)$ be a set of parameters, and let $e\in K$. Then either there is no $\alpha \in k(Ce)\setminus k(C)$, or there is no $\gamma\in \Gamma(Ce)\setminus \Gamma(C)$.
\end{lemma}

\begin{proof} We use arguments which are developed in \cite[Section 2.5]{HHM} (see Lemma 2.5.5) in the context of ACVF, and in a more general $C$-minimal context in \cite{HK}.

Suppose that $e$ realises the generic type $p_s|C$ of a $C$-definable open ball $s$ or chain of $C$-definable balls with no least element. By \cite[Lemma 3.19]{HK}, $p_s$ is orthogonal to the generic type of the closed ball $\calO$. That is, 
if $a$ is generic in $\calO$ over $C$, it is generic in $\calO$ over $Ce$. However, if $\alpha\in k(Ce)\setminus k(C)$, then  $\alpha={\rm res}(a)$ for some $a$, so $a$ is not generic in $\calO$ over $Ce$, but is generic in $\calO$ over $C$, which is  a contradiction.

On the other hand, if $e$ realises the generic type of a closed $C$-definable ball, then $\Gamma(Ce)=\Gamma(C)$ by Lemma~\ref{closedgeneric}.
 \end{proof}

We use the following notation to move between an imaginary and the corresponding subset of the field.
If $\alpha\in k(\Uu)$,  then $\alpha$ codes  a coset $A_\alpha=a+\mf{M}(\Uu)$ of $\mf{\Uu}$ for some $a\in \calO(\Uu)$.  Likewise, if $r\in RV(\Uu)$, then $r$ codes a coset $B_r=b(1+\mf{M}(\Uu))$ of $1+\mf{M}(\Uu)$.

\begin{proposition}\label{uniquetype}
Fix an algebraically closed set $C$ of parameters.

(i)  Let $r\in \RV$ with $v(r)\not\in\Gamma(C)$. Then
	$B_r$  realises a unique type in $VF$ over $C$, hence over $Cr$.

(ii) Let $\alpha\in k\setminus\acl(C)$. Then 
	$A_\alpha $ realises a unique type in $VF$ over $C$, and hence over $C\alpha$.

\end{proposition}

\begin{proof}
These follow from $C$-minimality of $T^{\an}_D$.
\end{proof}

 \medskip

\noindent
{\bf Proof of Theorem~\ref{theorem1}.}
We give a proof below in $(0,0)$-characteristic. For the small adaptation to mixed characteristic $(0,p)$, see Remark~\ref{mixed}.

First, observe that there is an $\emptyset$-definable isomorphism $\exp:(\mf{M},+)\to (1+\mf{M},.)$. This is given by the function symbol which is interpreted in the `standard' model $K_0$ by the usual power series function $f(X)=\Sigma_{i\geq 0} X^i/i!$. (Since the variable $X$ here ranges through the maximal ideal $\mf{M}$ rather than $\calO$,  so would be denoted by $\rho$ rather than $X$ in \cite{L}, there is indeed such a function symbol in $\calL^{\an}_D$.) We shall denote the inverse function $1+\mf{M} \to \mf{M}$, which is also $\emptyset$-definable, by $\log$.

Assume that $M$ is a small elementary submodel of $\Uu$, and let $m\in VF(M)$ with $\beta:=v(m)>0$. Define
$W:=\calO/\beta\mf{M}$, a definable set in $\Uu^{\eq}$. For $w\in W$ let $A_w$ be the corresponding coset of $\beta\mf{M}$. 
For any parameter set $C$, we shall say that $w$ is {\em generic} in $W$ over $C$ if there is no $\acl(C)$-definable proper sub-ball $B$ of $\calO$ with $A_w\subseteq B$. This agrees with the terminology of \cite[Definition 2.3.4]{HHM} for 1-torsors.
For any $\gamma\in \Gamma$, let $U_\gamma:=\{r\in RV: v(r)=\gamma\}$.
Let $E:(\beta\mf{M},+)\to (1+\mf{M},.)$ be the group isomorphism $E(x)=\exp(m^{-1}x)$. 

Choose $r$  in $RV$ with $\gamma:=v(r)$ infinite with respect to $M$, and choose 
$w\in W$ generic over $M\cup\{r\}$. Write $\bar{w}$ for ${\rm res}(x)$ where $x\in A_w$, and for $\alpha \in k$, let
$W_\alpha:=\{w\in W: \bar{w}=\alpha\}$.

For $a\in A_w$ and $b\in B_r$,
define the affine homomorphism $h=h_{a,b}:A_w \to B_r$, with homogeneous component $E$, by
$$h(x)=bE(x-a)=b\exp(m^{-1}(x-a)).$$
To prove the theorem, it suffices to  show that $\lceil h\rceil$ is not coded in $\calL^{\an}_{D,\G}$. First observe that $\lceil h\rceil\not\in \acl(M,w,r)$. For $w,r$ belong respectively to the strongly minimal sets $\red(a+\beta\calO)$ and $U_\gamma$, and by $C$-minimality arguments all elements of $\red(a+\beta \calO)\times U_\gamma$ realise the same type over $M$. Hence, by $C$-minimality again, if $b,b'\in B_r$ are distinct then
$\tp(a,b/M,w,r)=\tp(a,b'/M,w,r)$, whilst $h=h_{a,b}\neq h_{a,b'}$.

Suppose for a contradiction that there is a finite tuple in $\G$ which is a code for $\lceil h\rceil$ over $M$. 
Since $w,r\in \dcl(\lceil h \rceil)$, we may suppose this code
 includes 
$w$ and $r$, so we can write it as $(w, r, \be)$, where $\be$ is a tuple in $\G$. 
Let $\be=(\be_1,\be_2)$ where $\be_1$ is a tuple from $VF$ and $\be_2$ is a tuple from the geometric sorts other than $VF$.

\medskip

{\em Claim 1.} 
(i) $\be_2\in \acl(M,w,r,\be_1)$.

\qquad (ii) $\dim(\be_1/M)=1$.

{\em Proof of Claim.} (i) Suppose that $\be_2\not\in \acl(M,w,r,\be_1)$.

 Let $C' \subset K$ with 
$M \subset C'$ and  $w, r \in \dcl(C')$, such that there is an affine homomorphism
$g:A_w\to B_r$, also with homogeneous component $E$, defined over $C'$. 
Using an automorphism over $M,w,r$ if necessary, we may suppose that 
$\lceil h \rceil \not\in \acl(C')$.
Now $g(x)=b'\exp m^{-1}(x-a')$ for some 
 $a'\in A_w$, $b'\in B_r$. Then the function $\log(g/h):A_w\to \M$ is defined, and by properties of
 the exponential and logarithmic functions it satisfies 
$$
\log(g/h)(x) = \log(b'/b) + m^{-1}(a-a')= d\in \M.
$$
Hence, as $\lceil g\rceil\in\dcl(C')$ and  $g(x)=h(x)\exp(d)$, the element 
 $\lceil h \rceil$ is coded over $C'$ by the field element $d$. In particular, $d\not\in \acl(C')$.

We may choose $C'$ as above so that in addition $\be_2\not\in \acl(C',\be_1)$. Now $\lceil h\rceil$ is interdefinable over $C'$ with both $\be$ and $d$, so we find $d\in \dcl(C',\be_1,\be_2)\setminus \acl(C',\be_1)$. That is, there is a definable map from a product of sorts other than $VF$ to $VF$, with infinite range. This contradicts Lemma~\ref{finiteimage}.

(ii) 
We may choose $C'$ as in (i) so that $\dim(\be_1/M,w,r)=\dim(\be_1/C')$. Thus, as $\be_1\in \dcl(C',d)$, we have $\dim(\be_1/M,w,r)=1$. The result now follows from Lemma~\ref{finiteimage}.  

\smallskip

We  may now choose $e\in \bar{e}_1$ so that $\bar{e}_1$ is algebraic over $M \cup\{e\}$. The argument breaks into two cases.

\medskip
\noindent
{\em Case 1.} $\bar{w}\not\in \acl(M,e)$. Now by orthogonality of $k$ and $\Gamma$,
$\bar{w}\not\in \acl(M,e,v(r))$. 

{\em Claim 2.} Then $w\not\in \acl(M,e,r)$.

{\em Proof of Claim.} Suppose for a contradiction that $w\in \acl(M,e,r)$.   
Suppose that $\phi(x,y)$ is an $\calL^{\an,\eq}$-formula over $M \cup \{e\}$ such that $\phi(w,r)$ holds and $\phi(W,r)$ is finite.
By an easy $C$-minimality argument, there is $n_\phi$ such that for any $y\in U_\gamma$, if $\phi(W,y)$ is finite then it has at most $n_\phi$ elements. Let $\phi^*(x,y)$ be the formula
$\phi(x,y)\wedge |\phi(W,y)|\leq n_\phi$. Similarly, there are formulas $\bar{\phi}^*(z,y)$
and $\bar{\phi}(z,y)$ (with $z$ ranging over the residue field $k$, and $y$ over $RV$) such that $\bar{\phi}(k,r)$ is finite and contains $\bar{w}$, and $\bar{\phi}^*(\bar{w},r)$ holds, and such that $\bar{\phi}^*(z,y)$ expresses that $z\in k$ is algebraic over $M,e,y$ via $\bar{\phi}$.
As $k$ and $U_\gamma$ are strongly minimal and $\bar{w}\not\in \acl(M,e,\gamma)$, there are finitely many $r'\in U_\gamma$ such that $\bar{\phi}^*(\bar{w},r')$.  For $\alpha\in k$, let
$$S_\alpha:=\{x\in W: \exists y(\phi^*(x,y)\wedge {\rm res}(x)=\alpha\wedge \bar{\phi}^*(\alpha,y)) \}.$$
Then $S_{\bar{w}}$ is a finite non-empty subset of the infinite set $W_{\bar{w}}$ which is definable over $M,e,\gamma, \bar{w}$. Since $\bar{w}\not\in \acl(M,e,\gamma)$, 
$\bigcup_{\alpha\in k} S_\alpha$ is a definable subset of $W$ which is not a boolean combination of balls, contrary to $C$-minimality.

\smallskip
\noindent
{\em Case 2.} $\bar{w}\in \acl(M,e)$.

{\em Claim 3.} Then $r\not\in \acl(M,e,w)$.

{\em Proof of Claim.} 
By Lemma~\ref{weightone}, $\gamma\not\in \acl(M,e)$, and indeed, $\Gamma(M,e) = \Gamma(M)$. Thus, we may suppose $w\not\in \acl(M,e)$, 
as otherwise the claim follows immediately; for if $r\in \acl(M,e,w)=\acl(M,e)$ then $\gamma=v(r)\in \acl(M,e)$. We shall show that 
$\gamma\not\in\acl(M,e,w)$, which implies that $r\not\in\acl(M,e,w)$. So suppose that $\gamma\in\acl(M,e,w)$, so as $\Gamma$ is totally ordered
there is an $\acl(M,e)$-definable function $f$ with $f(w)=\gamma$. 

Let $B$ be the intersection of the closed $\Uu$-definable balls in $\calO$ containing $f^{-1}(\gamma)$. By o-minimality of $\Gamma$,
$B$ is a ball, of radius $\delta$, say, and is closed. Also, $\gamma$ is infinite with respect to $\Gamma(M,e)=\Gamma(M)$, but 
$\delta$ is bounded above in $\Gamma(M)$ by $\beta$. By Lemma~\ref{stablyembedded}, $\Gamma$ has the structure of a divisible ordered abelian 
group, and is stably embedded. Hence definable functions $\Gamma\to \Gamma$ are piecewise linear, so as $\delta\in \acl(M,e,\gamma)$, we have $\delta\in \acl(M,e)$. Also, $\lceil B \rceil \in \acl(M,e,\gamma)$.

Suppose first that $f^{-1}(\gamma)$ meets some element $s$ of $\red(B)$ in a non-empty  set which is not generic in $s$ over $M,e,\gamma,s$. Then let $B'$ be 
the smallest closed ball containing $f^{-1}(\gamma)\cap s$, and replace $B$ by $B'$. Note that $B'$ is still algebraic over 
$M,e,\gamma$, as there can be only finitely many such $s\in \red(B)$, by $C$-minimality. This process must terminate after
finitely many steps (again by $C$-minimality, and definability of $f^{-1}(\gamma)$), so for convenience we may suppose that for all $s\in\red(B)$, $f^{-1}(\gamma)\cap s$ is empty or generic in $s$ 
over $M,e,\gamma,s$.

If $\lceil B\rceil \in\acl(M,e)$, then as $\gamma\notin\acl(M,e)$, $f^{-1}(\gamma)$ meets just finitely many elements of $\red(B)$,
including some non-algebraic element. For each $s\in\red(B)$, define $F(s)$ to be the (fixed)  value of $f$ on generic elements of $s$. We obtain
a definable function $F:\red(B)\to \Gamma$ with infinite range, which contradicts Lemma~\ref{orthog}.

Thus we may assume that $\lceil B\rceil \notin\acl(M,e)$. For any set $S$ and value $\epsilon$, write $S/\epsilon$ for the set of 
closed sub-balls of $S$ of radius $\epsilon$.
It follows from $C$-minimality that there is a ball $B''$ containing $B$ and of radius $\delta''<\delta$, such that all elements of 
$B''/\delta$ have the same type. If $f^{-1}(\gamma)$ meets infinitely many elements of $\red(B)$, then $f^{-1}(\gamma)$ is a 
generic subset of $B$ (abusing notation -- we here view $B$ as a subset of $W$), and we have an induced function $F:B''/\delta \to \Gamma$ with $F(B)=\gamma$. Pick $\delta'\in \Gamma$
with $\delta'' < \delta' < \delta$. For each $B'\in B''/\delta'$, let 
$$
	S(B') = \{ F(B^*) : B^*\in B'/\delta\}.
$$
The sets $S(B')$ form a family of infinite uniformly definable pairwise disjoint subsets of $\Gamma$, which contradicts the o-minimality of $\Gamma$.

Thus we may also assume that $f^{-1}(\gamma)$ meets only finitely many elements of $\red(B)$, and meets each in a generic set.
As $B\notin \acl(M,e)$, all elements of $\red(B)$ are generic so have the same type  over $\acl(M,e,\lceil B\rceil)$. It follows that there is a definable function
$g:\red(B) \to \Gamma$, where $g(s)$ is the generic value of $f$ on $s$. Now $g$ has infinite range. As $\red(B)$ is in definable bijection with $k$, this contradicts 
Lemma~\ref{orthog}. 

\smallskip

To complete the proof of the theorem, suppose first that Case 1 holds. 
Choose $c'\in B_r$ generic over $\acl(M,e,r,w)$. Then by Claim 2,
$w\not\in\acl(M,e,r,c')$. Hence $A_w$ realises a single 1-type over $\acl(M,e,r,w,c')$ by Proposition~\ref{uniquetype} (ii), and in particular 
$A_w\cap \acl(M,e,r,w,c')=\emptyset$. But as  $\lceil h \rceil\in \acl(M,e,r,w)$,
$h^{-1}(c')\in \acl(M,e,r,w,c')\cap A_w$, which is a contradiction.

If Case 2 holds,
 we argue as in Case 1, with $r$ and $w$ reversed. Choose  $c$
generic  in $A_w$ over $\acl(M,e,w,r)$. Then by  Claim 3, $r\not\in \acl(M,e,w,c)$, so the elements of $B_r$ all have the 
same type over $\acl(M,e,w,r,c)$ by Proposition~\ref{uniquetype} (i), so $B_r \cap \acl(M,e,w,r,c)=\emptyset$. This however is impossible, as  
$\lceil h\rceil \in \acl(M,e,w,r)$, and $h(c)\in B_r$. \hfill $\Box$

\begin{remark}\label{mixed}\rm
For the case when $K_0$ has mixed characteristic $(0,p)$, that is, ${\rm char}(k)=p$, a slight adaptation of the definition of the 
exponential function is needed. First, suppose $p>2$. Then the power series $f(X)=\Sigma_{i\geq 0} \frac{p^i}{i!}X^i$ is defined on 
$\calO$ and is a power series in the valuation ring sort in the ring of separated power series in the sense of Lipshitz. This gives a 
definable isomorphism $\exp:(p\calO,+)\to (1+p\calO,.)$, given by
$\exp(px)=f(x)$. The argument now proceeds as above, with $\beta>1=v(p)$. In the case $p=2$, the function $f$ has form
$f(X)=\Sigma_{i\geq 0}\frac{p^{2i}}{i!}X^i$, and defines an isomorphism $\exp: (p^2\calO,+)\to (1+p^2\calO,.)$ given by 
$\exp(p^2x)=f(x)$.
\end{remark}

\section{The real closed case.}

 As in the Introduction, we let $\bar{{\mathbb R}}$ denote a polynomially bounded o-minimal expansion of the real field, in  a language $\bar{\calL}$, in which restricted exponentiation is definable. Let
$T=\Th(\bar{{\mathbb R}})$, $\bar{F}$ a non-archimedean model of $T$, $V$ the convex subring of 
$\bar{F}$ consisting of the finite elements, and $\mu$ its ideal of infinitesimals. We view $(\bar{F},V)$ 
as a structure (a model of $T_{\con}$) in the language $\bar{\calL}_{\con}=\bar{\calL}\cup\{P\}$, where $P$ is a unary predicate interpreted here by $V$.
It is shown in \cite{DL} that $T_{\con}$ is  weakly o-minimal; in fact, weak o-minimality follows from \cite{BP}, where it is shown that any expansion of an o-minimal structure by a predicate for a convex subset has weakly o-minimal theory. Furthermore, by \cite[(3.10)]{DL}, if $T$ has quantifier elimination and is universally axiomatised, then $T_{\con}$ has quantifier elimination.

As in Section~2, $\Uu$ denotes a large saturated model of $T_{{\rm con}}$, whose underlying field has domain $K$.
Now exponentiation is defined in $\bar{{\mathbb R}}$, and hence in $\bar{F}$, on any set of
 form $[-n,n]$ for $n\in {\mathbb N}^{>0}$. Furthermore, if $x\in \mu$ then $\exp(x)\in 1+\mu$,
 so in $\Uu$, $\exp(\mf{M})\subseteq 1+\mf{M}$. In fact, the restriction $\exp|_{\mf{M}}:\mf{M}\to 1+\mf{M}$ is bijective,
 with inverse the map $\log|_{1+\mf{M}}$. 

The proof of Theorem~\ref{theorem1}, with the same imaginary $\lceil h \rceil$, yields also a proof of Theorem~\ref{theorem2}, and below we only pay attention to points of difference. First, we give an analogue of Lemma~\ref{exchange}.
\begin{lemma} \label{exreal}
(i) Let $C\subset \Uu\models T_{\con}$, and $b\in \acl(C)$. Then $b$ is algebraic over $C$ in the reduct of $\Uu$ to $\bar{\calL}$.

(ii) Algebraic closure has the exchange property in models of $T_{\con}$.
\end{lemma}

\begin{proof} (i) First observe that $T$, being the theory of an o-minimal expansion of an ordered field, has definable Skolem functions. Hence, by extending the language $\bar{\calL}$ by definitions, we may suppose that $T$ is universally axiomatised and has quantifier elimination. Thus,
 by \cite[3.10]{DL}, $T_{\con}$ has quantifier elimination, and so,  there is a quantifier-free formula $\phi(x,\bar{y})$ over $\bar{\calL}_{\con}$, and $\bar{a}\in C^{l(\bar{y})}$, such that
 $\phi(\Uu,\bar{a})$ is finite and contains $b$. We may suppose that $\phi(x,\bar{y})$ is a conjunction of atomic and negated atomic formulas, and in particular that it has the form
 $$\psi(x,\bar{y}) \wedge \bigwedge_{i=1}^s (t_i(x,\bar{y})\in \calO)
 \wedge \bigwedge_{i=1}^{s'} (t_i'(x,\bar{y})\not\in \calO),$$
where $\psi(x,\bar{y})$ is a quantifier-free $\bar{\calL}$-formula and the $t_i,t_i'$ are terms.
We may also suppose that $\psi(\Uu,\bar{a})$ is infinite, and indeed that $\psi(\Uu,\bar{a})$ is an open interval, and that the $t_i(x,\bar{a})$ and $t_i'(x,\bar{a})$ are continuous at $b$ (since the points of discontinuity will be algebraic over $\bar{a}$ in the  reduct to $\bar{\calL}$, by o-minimality). However, it is now impossible that $\phi(\Uu,\bar{a})$ is finite, by continuity of the $t_i,t_i'$. 

(ii) This follows immediately from (i), as algebraic closure has the exchange property in o-minimal theories. 
\end{proof}

Lemma~\ref{finiteimage} and Proposition~\ref{uniquetype} go through unchanged in the current setting. Our analogue of 
Lemma~\ref{stablyembedded} is the following.

\begin{lemma}\label{strongmin2}
In any model  of $T_{\con}$, the following hold.

(i) The residue field, equipped with the induced $\emptyset$-definable relations of $T_{\con}$, is an o-minimal structure and is stably embedded.

(ii) The value group $\Gamma$, equipped with the induced $\emptyset$-definable relations of $T_{\con}$, is an o-minimal structure and is stably embedded, and has the structure of an ordered vector space over an archimedean ordered field, its `field of exponents'. 
\end{lemma}

\begin{proof} 
By \cite[Theorem A, Corollary 1.12]{D}, the structure induced on the residue field is o-minimal (and elementarily equivalent to $\bar{{\mathbb R}}$).  
For the o-minimality of the value group, see \cite[4.5]{DL}. For the description of its structure, see Theorem B and (3.1) of \cite{D}. 

In both cases, stable embeddedness follows from \cite[Theorem 2]{ho}, or from the main theorem of \cite{pill}.
\end{proof}

\begin{lemma}\label{orthog2}
In any model of $T_{\con}$, the residue field and value group are orthogonal in the sense of Proposition~\ref{orthog}.
\end{lemma}

\begin{proof} This follows from \cite[Proposition 5.8]{D}. 
\end{proof}

\begin{lemma} \label{model}
Let $C=\acl(C)\cap K$, and suppose there is $c\in C$ with $v(c)\neq 0$. Then $C\models T_{\con}$.
\end{lemma}

\begin{proof} Since $T$ has definable Skolem functions, the reduct $C|_{\bar{\calL}}$ of $C$ to $\bar{\calL}$ is a model of $T$. By \cite[3.13]{DL}, the theory of expansions of models of $T$ by a predicate for a proper convex valuation ring closed under $\emptyset$-definable continuous functions is complete. It follows that $(C|_{\bar{\calL}}, \calO\cap C)\models  T_{\con}$. 
\end{proof}

We now give  analogues of Lemmas~\ref{closedgeneric} and \ref{weightone}.

\begin{lemma}\label{closedgeneric2}
Let $C$ be a parameter set, $s$ a $C$-definable closed ball, and let $a\in K$ be an element of $s$ which is not in any  $C$-definable
 proper sub-ball of $s$. Then $\Gamma(C)=\Gamma(C,a)$.
\end{lemma}

\begin{proof} Suppose $\gamma\in \Gamma(C,a)\setminus \Gamma(C)$, so there is a $C$-definable function $f$ with $f(a)=\gamma$. 
Write $\red(s)$ for the set of open sub-balls of $s$ of the same radius as $s$. Let $u$ be the element of $\red(s)$ containing $a$.
Then $T_u :=\{ f(x): x\in u\}$ is a finite union of points and open intervals in $\Gamma$. There is a boundary point of $T_u$ lying in 
$\Gamma(C,\lceil u \rceil)\setminus \Gamma(C)$, contradicting Lemma~\ref{orthog2}.
\end{proof}

\begin{lemma}\label{weightone2}
 Let $C=\acl(C)$ be a set of field parameters containing an element with non-zero value, and let $e\in K$. Then either there is no $\alpha \in k(C,e)\setminus k(C)$, or there is no $\gamma\in \Gamma(C,e)\setminus \Gamma(C)$.
\end{lemma}

\begin{proof} Suppose $\alpha,\gamma\in \acl(C,e)\setminus \acl(C)$, with $\alpha\in k$ and $\gamma \in \Gamma$. Then by Lemma~\ref{model}, there is a field element $e'\in \acl(C,e)$ with residue $\alpha$. By Lemma~\ref{exreal}(ii), as $e'\not\in\acl(C)$, we have $e\in \acl(C,e')$, so 
$\gamma\in \dcl(C,e')$. Hence there is a $C$-definable function $f:K \to \Gamma$ such that $f(e')=\gamma$. Let $\Delta:=\{f(x): x\in A_\alpha\}$, where, as in Section 2,  $A_\alpha:=a+\mf{M}(\Uu)$ for some $a$ with residue $\alpha$. Then $\Delta$ is an infinite subset of $\Gamma$, since otherwise there is a definable function $k\to \Gamma$ with infinite range, contrary to Lemma~\ref{orthog2}. However,   as $k$ and $\Gamma$ are orthogonal, and $\Delta$ is 
coded by a tuple from $\Gamma$, $\Delta$ is $C$-definable. It follows that if $\delta\in \Delta$, then $f^{-1}(\delta)$ contains a proper non-empty subset of infinitely many residue classes, contrary to weak-o-minimality. 
\end{proof}

\medskip
\noindent
{\bf Proof of Theorem~\ref{theorem2}.}
The  proof of Theorem~\ref{theorem2} proceeds as in Section 2. We choose $w,r,a,b$ and define $h_{a,b}$ as before, using the same notation, and aim for a contradiction from the assumption that $h_{a,b}$ is coded in the geometric sorts. The 
only substantial change is in Case 2 (in Case 1, the use of strong minimality is easily replaced by an o-minimality argument). We describe this case in some detail, and leave the rest of the proof to the reader. For any 
subset $S$ of $K$ and value $\epsilon\in \Gamma$, write
$(S/\epsilon)^o$ for the set of open sub-balls of $S$ of radius $\epsilon$, and $(S/\epsilon)^c$ for the closed ones.

As in the algebraically closed case (Case 2), working under the assumption that $\bar{w}\in\acl(M,e)$, we shall show $\gamma\not\in \dcl(M,e,w)$.
Notice that $\Gamma(M,e)=\Gamma(M)$, by Lemma~\ref{weightone2}, so in particular $\gamma$ is infinite with respect to $\Gamma(M,e)$.
So suppose for a contradiction that there is an $\acl(M,e)$-definable function $f:W\to\Gamma$ with $f(w)=\gamma$.

Observe that there is a natural  ordering, induced by the field ordering and also denoted $<$, on $W$. By weak o-minimality of $M$, the set $f^{-1}(\gamma)$ is a finite 
union of maximal convex subsets $D_1,\ldots,D_t$ of $(W,<)$, with $D_1<\ldots< D_t$. First suppose that all of the $D_i$ are finite (so by density are singletons), and let $B$ be the
closed ball of radius $\beta$ containing $D_1$. Then $f$ defines a function from  $(B/\beta)^o$ to $\Gamma$
with infinite range, contradicting Lemma~\ref{orthog2}.

Thus we may assume some $D_i$, say $D_1$, is infinite. Put $D:=D_1$. We sometimes abuse notation and identify $D$ with the union of the 
balls (elements of $W$) in $D$, viewed as a subset of $K$.
Let $B$ be the intersection of the $\Uu$-definable closed sub-balls of $\calO$ containing $D$. By o-minimality of $\Gamma$, $B$ is itself a ball, of radius $\delta$, say; also, $B$ is closed. As usual, let $\red(B)$ be the set of open sub-balls of $B$ of the same radius as $B$. Observe that $\red(B)$ is o-minimal, by Lemma~\ref{strongmin2}, since it is in definable bijection with $k$.

We have $\delta\in \dcl(M,e,\gamma)$, and $\delta\leq \beta\in \Gamma(M)$.  However, $\gamma$ is infinite with respect to $\Gamma(M)$.  Since $\Gamma$ is stably embedded and has the structure of an ordered vector space over an archimedean ordered field (see Lemma~\ref{strongmin2}), it follows that $\delta\in \dcl(M,e)$.

Suppose first that $\lceil B\rceil\in\acl(M,e)$. If $D$ contains an element $s\in \red(B)$ not in $\acl(M,e)$, then $\Gamma(M,e,s)=\Gamma(M,e)$ by Lemma~\ref{closedgeneric2}. As $\gamma\in \Gamma(M,e,s)$, we then find $\gamma\in\Gamma(M,e)$, a contradiction. Thus, $D$ meets just finitely many elements of $\red(B)$, say $s_1,\ldots,s_n$, and these are algebraic over $M,e,\gamma$ and hence over $M,e$. As $D$ is convex and $\red(B)$ is densely ordered, this forces $n=1$. Also, $D$ meets $s_1$ in an initial or final segment of $s_1$ (with possibly $D=s_1$). Now as $\gamma$ is the initial or final value of $f$ on $s_1$, and $s_1\in \acl(M,e)$, we find $\gamma\in\acl(M,e)$, again a contradiction. 

Thus, $\lceil B\rceil\not\in \acl(M,e)$, so by weak o-minimality, $\red(B)$ is a complete type over $\acl(M,e,\lceil B\rceil)$. From this, the definition of $B$, and weak o-minimality it follows easily that for each $s\in \red(B)$, $D$ contains $(s/\beta)^o$ or is disjoint from $(s/\beta)^o$. We claim that $D=B$ or $D\in \red(B)$. For if not, 
then using the fibres of $f$ we can partition $\red(B)$ into a uniformly definable family of 
infinite disjoint convex sets. As $\red(B)$ is in definable bijection with $k$, this contradicts the o-minimality of $k$.

If $D\in \red(B)$ then $f$ induces a definable function $\bar{f}:\red(B)\to \Gamma$ with infinite range, contrary to Lemma~\ref{orthog2}. Thus, $D=B$. Since $\lceil B\rceil\not\in \acl(M,e)$, there is an interval $I$ of $(\calO/\delta)^c$ containing $B$ in its interior such that all elements of $I$ have the same type over $\acl(M,e)$. Pick $B_1,B_2\in I$ with $B_1<B<B_2$, and $b_i\in B_i$, and put $\delta'':=v(b_1-b_2)$. Let $B''$ be the open ball of radius $\delta''$ containing $B$, so $(B''/\delta)^c\subset I$.

For each $B_0\in (B''/\delta)^c$, $f$ takes constant value, $F(B_0)$, say, on $(B_0/\beta)^o$. Now all elements of $(B''/\delta)^c$ have the same type over $\acl(M,e)$. It follows that $F:(B''/\delta)^c\to \Gamma$ is injective. Let $\delta'\in \Gamma$ with $\delta''<\delta'<\delta$. For $B'\in (B''/{\delta'})^o$, define $S(B'):=\{F(B_0): B_0\in (B'/\delta)^c\}$. Then the sets $S(B')$ (for $B'\in (B''/{\delta')^o}$), form an infinite uniformly definable family of pairwise disjoint subsets of $\Gamma$, contrary to the o-minimality of $\Gamma$. \hfill $\Box$

\section{The $p$-adic case.}

 Fix a prime $p$. We shall assume below that $p>2$ to ensure that $p$-adic exponentiation converges on $p{\mathbb Z}_p$ -- see Remark~\ref{even} for the remaining case. We shall work with the language $\calL^{\an}_p$ and theory $T^{\an}_p$, as in the Introduction. Let 
$\Uu\models T^{\an}_p$ 
be sufficiently saturated. Observe first that the power series 
$G(X)=\Sigma_{i=0}^\infty \frac{p^nX^n}{n!}$ which defines the exponential map corresponds 
to a function symbol of $\calL^{\an}_p$, since $v_p(p^n/n!) \to \infty$ as $n \to \infty$.
We have $\exp(x)=G(xp^{-1})$ for $x\in p{\mathbb Z}_p$. The exponential map induces (in the standard model ${\mathbb Q}_p^{\an}$) a group isomorphism from $(p{\mathbb Z}_p,+)$ to $(1+p{\mathbb Z}_p,.)$ with inverse the natural logarithm. Indeed, there is a  power series $H(X)=\Sigma_{i=1}^\infty
\frac{(-1)^{n+1}p^nX^n}{n}$, given by an $\calL^{\an}_p$-symbol,  such that for $x=1+py\in 1+p{\mathbb Z}_p$ we have
$\log(x)=\log(1+py)=H(y)$.

We summarise   some facts about $T^{\an}_p$ used below.
Following \cite{HM2}, we shall say that an expansion $M$ of a model of $\Th({\mathbb Q}_p)$ is {\em $P$-minimal} if, for every $N\equiv M$, every parameter-definable subset of the field (in one variable) is quantifier-free definable in Macintyre's language $\calL_p$, so is semi-algebraic.

\begin{proposition} \label{padicfacts}
(i) $T^{\an}_p$ is $P$-minimal.

(ii) Algebraic closure has the exchange property in models of $T^{\an}_p$.

(iii) $T^{\an}_p$ has definable Skolem functions.

(iv) Let $\Uu\models T^{\an}_p$, and let $f:X\to K$ be a definable map from a geometric sort other than $K$ to $K$. Then $f$ has finite image.

(v) The value group $\Gamma$, equipped with $\emptyset$-definable induced structure, is a model of Presburger arithmetic expanded by constants, and is stably embedded.

\end{proposition}

\begin{proof}
(i) This is the main theorem (Theorem A) of \cite{DHM}.

(ii) By \cite[Theorem 6.2]{HM2}, algebraic closure in any $P$-minimal theory
has the exchange property, so this follows from (i).

(iii) See \cite[3.6]{DvdD}. 

(iv) The proof of Lemma~\ref{finiteimage} applies here. 

(v) This follows from (i) above in combination with Theorems 5 and 6 of \cite{cl0}.
\end{proof}

We shall work over an $\omega$-saturated model $M$ of $T^{\an}_p$, inside the sufficiently saturated model $\Uu$. 

Let $\Oo(M)$ be the valuation ring of $M$, and fix  $\beta \in \Gamma(M)$ which is infinite in the sense that for all $n\in {\mathbb N}$, $\beta>nv(p)$.  Let $W= \Oo(M) / \beta\Oo(M)$.   
  For $w \in W$, let $A_w$ denote the corresponding additive coset of $\beta \Oo(\Uu)$.   
  For 
  $\gamma \in \Gamma$ let
 $V_\gamma$ be the annulus $\{x: v(x)=\gamma\}$, and for 
 $r\in RV$ let $B_r$ 
denote the corresponding multiplicative coset of $1+ \mf{M}$ (viewed as a subset of $\Uu$).
Clearly, for any base $C$  with $\beta \in C$, if $w$ is a non-algebraic element of $W$, then the subset $A_w$ of $VF$ has no $C$-algebraic points; and similarly if $\gamma \in \Gamma \setminus \acl(C)$ then the annulus $V_\gamma$ has no  $C$-algebraic points.  

There is an ${\rm Aut}(\Uu/M)$-invariant partial type $q$ of elements of $\Gamma$ determined by the formulas $x>\gamma$ for
 all $\gamma\in \Gamma(\Uu)$. 
Note also that the map $v:RV \to \Gamma$ is finite-to-one.

\begin{lemma} \label{Pmin}
Let $U$ be an infinite definable subset of $W$. Then there is an infinite subset $U'$ of $U$ 
and a ball $B$ such that $U':=\{w\in W:A_w\subset B\}$.
\end{lemma}

\begin{proof} This is an easy consequence of $P$-minimality. Indeed, by $P$-minimality
the set $VF(U):= \bigcup_{w\in U} A_w$ is a finite union of sets of the form
$$D=\{ x\in VF: \gamma_1 \Box_1 v(x-a) \Box_2 \gamma_2 \,\&\, P_n(\lambda(x-a))\},$$
where $\Box_1,\Box_2\in \{<,\leq\}$, $a\in VF$, $\gamma_1,\gamma_2\in \Gamma \cup\{-\infty,\infty\}$, and $\lambda$ is chosen from a fixed set of coset representatives of $P_n$ in $K^*$.
 We may suppose that $VF(U)$ is exactly this set $D$.
If there is no  ball $B$ as in the lemma, then there is $\beta'<\beta$ such that $\beta-\beta'$ 
is finite and the set $VF(U)$ is a union of a set of maximal subballs of $VF(U)$ which are pairwise
 disjoint and have radius between $\beta'$ and $\beta$. However, this is impossible, for  if the set $D$ is a
 union of infinitely many disjoint maximal balls, then they must have infinitely
 many different radii. This can be seen for example from the following standard lemma, proved in  \cite[Lemma 2.3]{HM2}.
\end{proof}

\begin{lemma}\cite[Lemma 2.3]{HM2}\label{hmlemma}
Let $x,x',a,\lambda\in {\mathbb Q}_p$ with $\lambda\neq 0$, and let $n>1$ be an integer. If $P_n(\lambda(x-a))$ holds, and
$v(x-x')>2v(n)+v(x-a)$, then $P_n(\lambda(x'-a))$ holds. 
\end{lemma}

\begin{lemma}\label{fincover}
Let $\rho:\Gamma'\to \Gamma$ be a definable finite cover of $\Gamma$, so $\Gamma'$ is a definable set and $\rho$ a definable function with  fibres of sizes uniformly bounded by $t\in {\mathbb N}$.
Then there is no definable partial function $\Gamma' \to W$ with infinite range.
\end{lemma}

\begin{proof} Suppose that $\Gamma_0\subset \Gamma$ is infinite, and that there is $f: \Gamma_0' :=\rho^{-1}(\Gamma_0)\to W$
 with infinite range. Put $U:=\ran(f)$. Replacing $\Gamma_0'$ by an infinite definable subset if necessary, we may by 
 Lemma~\ref{Pmin} suppose that there is  a ball $B$ of radius $\delta<\beta$ (with $\beta-\delta$  infinite) such that
$U=\{u\in W: A_{u}\subset B\}$. Choose $j$ least such that $p^j>t$, and put $\beta':=\beta-t$. Let $W':=\Oo/\beta'\Oo$. 

{\em Claim.} We may suppose that for each $u\in U$, $\rho(f^{-1}(u))$ has a least element.

{\em Proof of Claim.} By $P$-minimality, there is $N\in {\mathbb N}$ such that any definable subset of $\Gamma$ is a Boolean 
combination of intervals and cosets of finite index subgroups $n\Gamma$ where $n\in {\mathbb N}$ with $0<n<N$. Since the sets $\rho (f^{-1}(u))$ 
form an infinite  uniformly definable family of subsets of $\Gamma$ such that any intersection of any $t+1$ members of the family is empty, 
we may (reducing $U$ if necessary) suppose these sets  are all bounded below in $\Gamma$. As $\Gamma$ is a ${\mathbb Z}$-group 
(a model of Presburger arithmetic), it is clear that any definable subset of $\Gamma$ which is bounded below has a least member, 
yielding the claim.

\smallskip

For each $u\in U$, let $g(u):=\Min \rho (f^{-1}(u))$, so $g:U\to \Gamma$ is a definable  function with fibres of size at most $t$. 
For $u,u'\in U$, put $u\sim u'$ if there is $w'\in W'$ such that $A_u\cup A_{u'}\subseteq A_{w'}$; so $\sim$ is an equivalence relation with classes of size $p^j>t$. Finally, say that $u\in U$ is {\em good} if $g(u)\leq g(u')$ for all $u'\in U$ such that $u \sim u'$. 
Then each $\sim$-class of $U$ has some good elements and some elements which are not good.
It follows that $\{u\in U: u\mbox{~is good}\}$ is an infinite definable subset of $W$ which does not satisfy the conclusion of Lemma~\ref{Pmin}, a contradiction. 
\end{proof}

We choose $w\in W\setminus \acl(M)$, and  an element $r$ of $RV$ with $v(r) \models q| \acl(M,w)$. Put $\gamma:=v(r)$. By the description of 1-variable definable sets and Lemma~\ref{hmlemma}, all elements of $A_w$ realise the same type over $M$.

\begin{lemma}\label{padicalg} Let $e$ be a field element. 

(i)  If $w\not\in \acl(M,e)$, then $w \not\in \acl(M,e,r)$.

(ii)  If $w\in \acl(M,e)$, then $\gamma\not\in \acl(M,e,w)$.

\end{lemma}

\begin{proof} (i) Since $r\in \acl(\gamma)$, it suffices to show $w\not\in \acl(M,e,\gamma)$. 
So suppose that $w\in \acl(M,e,\gamma)$. Then for any $e'\in VF$ with $v(e')=\gamma$, 
$w\in \acl(M,e,e')$, so, by the existence of definable Skolem functions, $w\in \dcl(M,e,e')$, with say $w=f(\bar{m},e,e')$, where $\bar{m}$
is a tuple of field elements of  $M$. Define an equivalence relation $\equiv$ on an appropriate definable  subset $X$ of $K$, putting 
$x\equiv y$ if and only if $ v(x)=v(y)$ and $f(\bar{m},e,x)=f(\bar{m},e,y)$. Then $X/\equiv$ is a definable finite cover of a subset of 
$\Gamma$, and $f$ induces  a definable function $\bar{f}:X/\equiv \to W$ with infinite range. This contradicts Lemma~\ref{fincover}.

(ii) Suppose $\gamma\in \acl(M,e,w)$. As $w\in \acl(M,e)$, we have $\gamma\in \acl(M,e)$, 
so there is an $M$-definable function $f:K \to \Gamma$ with $f(e)=\gamma$. As $w\in \acl(M,e)$, which is a model,
there is a field element $e'\in w$ inter-algebraic with $e$ over $M$, so we may suppose $e\in w$. 

We consider the set $f^{-1}(\gamma)$.
Suppose first that $f^{-1}(\gamma) \subseteq w$. Then $A_w$ is a union of fibres of $f$.

If $f^{-1}(\gamma)$ is infinite, choose infinite $\beta'\in \Gamma$ with $\beta'<\beta$ and $v(\beta-\beta')$ infinite. For any ball $B$ write $f(B):=\{f(x):x\in B\}$. By Lemma~\ref{Pmin} (and with appropriate choice of $\beta'$), there is an infinite definable set  $S$ of balls $B\in \calO/\beta'\calO$ such that the set $\{f(B):B\in S\}$ is a uniformly definable infinite family of infinite definable pairwise disjoint subsets of $\Gamma$. This is impossible, by Proposition~\ref{padicfacts}(v) and quantifier elimination in Presburger arithmetic.

On the other hand, if $f^{-1}(\gamma)$ is a {\em finite} subset of $w$, 
then $f(w)$ is an infinite subset of $\Gamma$, and by taking a family of translates of $w$ (inside a larger ball) and applying $f$, we obtain again an infinite uniform family of infinite disjoint definable subsets of $\Gamma$.

Suppose now that $f^{-1}(\gamma)\not\subseteq w$. 
By $P$-minimality (Proposition~\ref{padicfacts}(i)), we may write $f^{-1}(\gamma)$ as a finite unions of `1-cells' $C_1,\ldots,C_t$ say. Here, the $C_i$ as before have the form
$$\{x\in VF: \gamma_1\Box_1 v(x-a)\Box_2 \gamma_2 \,\&\, P_n \lambda(x-a)\},$$
where $\Box_1,\Box_2\in \{<,\leq\}$, $a\in VF$, $\lambda$ is chosen from a fixed set of coset representatives of $P_n$ in $K^*$, and $\gamma_1,\gamma_2 \in \Gamma\cup\{-\infty,\infty\}$.

Let $B$ be the smallest ball containing $f^{-1}(\gamma)$. (Note here that
the intersection of the balls containing $f^{-1}(\gamma)$ is a ball, since any non-empty definable subset of $\Gamma$ which is bounded above has a greatest element -- see e.g. \cite[Lemma 4.4]{HM2}.) Let $\delta$ be the radius of $B$. Now $\delta<\beta\in\Gamma(M)$ and $\gamma>\Gamma(M)$. As $\Gamma$ carries the structure of Presburger arithmetic,  all definable functions $\Gamma \to \Gamma$ are piecewise linear over ${\mathbb Q}$(see also \cite[Proposition 2]{cl0}. Hence, as $\delta\in \dcl(M,\gamma)$, we have $\delta\in \dcl(M)$. 
We claim that $\lceil B \rceil\in \acl(M)$. Indeed, otherwise $\delta>n$ for all $n\in {\mathbb N}$ and we may argue as above (with $w$ replaced by $B$) to obtain a contradiction.

By Lemma~\ref{hmlemma}, considering the form of the $C_i$,
 there is a ball $B'$ with $s\subset f^{-1}(\gamma)$ such that $B'$ is at finite distance from $B$ in the natural tree structure on the set of all balls. (This is the graph whose vertex set is the set of all balls, with two balls $B_1,B_2$ adjacent if $B_1\subset B_2$ or $B_2 \subset B_1$ and there is no other ball strictly between them.) Thus $\lceil B'\rceil\in \acl(M,\lceil B\rceil)$. As $f$ takes constant value $\gamma$ on $s$, it follows that $\gamma\in \acl(M,\lceil B \rceil)=\acl(M)$, a contradiction.
\end{proof}

\medskip
 \noindent
 {\bf Proof of Theorem~\ref{theorem3}.}
To define the imaginary which cannot be coded in $\G$, choose $a\in A_w$ and $b\in B_r$, and let $c\in VF(M)$ with $v(c)=\beta$. 
There is a definable homomorphism $E:(\beta\Oo,+) \to (1+\mf{M}, .)$ given by
$E(x)= \exp(pc^{-1}x)$.
Define the affine homomorphism $h(a,b): A_w \to B_r$, with homogeneous component $E$, by 
$h_{a,b}(x)=b{\rm exp}(pc^{-1}(x-a))=bE(x-a)$. As usual,  we argue by contradiction from the assumption that $h$ is coded in the geometric sorts.

As before (using Proposition~\ref{padicfacts}(iv)), if $C'\supset M$ is a larger base containing $w,r$, chosen so that some affine 
homomorphism $g:A_w\to B_r$ with homogeneous component $E$ is definable over $C'$, then $h$ is coded over $C'$ by a fixed 
field element $d$. Again as before, $h$ is coded over $M$ by $(w,r,\bar{e})$ where $\bar{e}$ is a tuple of field elements of dimension 
1 over $M$. By the existence of Skolem functions, we may suppose that $\bar{e}$ is a single field element, 
say $e$. Thus, by Lemma~\ref{padicalg},  either $w\not\in \acl(M,e,r)$, or $w\in \acl(M,e)$, and $v(r)$ realises $q|\acl(M,e,w)$, so 
$r\not\in \acl(M,e,w)$. 

Each case is eliminated as in the algebraically closed case, the first as in Case 1, the second as in Case 2. $\Box$

\begin{remark} \label{even} \rm
For the case $p=2$, the exponential map converges on $p^2{\mathbb Z}_p$ and induces an isomorphism 
$(p^2{\mathbb Z}_p,+)\to (1+p^2{\mathbb Z}_p,.)$. The argument above can easily be adjusted to give a proof of 
Theorem~\ref{theorem3} in this case too. See also Remark~\ref{mixed}.
\end{remark}

\section{Alternative poof of theorem~\ref{theorem1}.}
In this final section we sketch  a slightly different proof of Theorem~\ref{theorem1}, at least in residue characteristic 0, with a rather simpler imaginary. The proof that it cannot be eliminated is more stability-theoretic, and takes place mainly in the residue field. It is not immediately clear whether the proof  can be adapted to yield Theorems~\ref{theorem2} and \ref{theorem3}. The argument given below uses parts of the proof in Section 2.

As in Section 2, let $K\models T^{\an}_D$ be sufficiently saturated, of residue characteristic 0. For $F$ a field, let $G_a(F)$ and $G_m(F)$ denote respectively the additive and multiplicative groups of $F$, and put $G:=G_a(K)\times G_m(K)$. By $G(\Oo)$ we shall mean the group of $\Oo$-rational points of $G$.  Also let $\pi:G(\Oo)\to G(k)$ be the residue map. Observe, using strong minimality of $k$, that 

(1) any infinite definable subgroup of $G(k)$ contains 
$G_a(k)\times \{1\}$ or $\{0\}\times G_m(k)$. 

\noindent
Indeed, otherwise there would be an isogeny $G_a(k)\to G_m(k)$, which is impossible (consider torsion in the two groups).

For $\alpha\in G(k)$, put $G_\alpha:=\{g\in G(\Oo): \pi(g)=\alpha\}$. For $i=1,2$ let $\pi_i$ denote the projection from $G(\Oo)$ to the $i^{{\rm th}}$ coordinate.
Let $E<G$ denote the graph of the exponential map $\exp:(\mf{M},+)\to (1+\mf{M},.)$. We aim to show that a generic coset of $E$ in $G(\Oo)$ (an affine homomorphism of torsors  with homogeneous component $\exp$, whose
 graph is a subset of 
  $G_\alpha$ for some generic $\alpha\in G(k)$), 
is not coded in $\G$.

As in the first proof of Theorem~\ref{theorem1}, we shall work over a small 
 elementary submodel $M$ of $\Uu$. Fix $\alpha \in G(k)$ and choose $C'\subset K$ such 
that $M\subset C'$, $\alpha\in  \dcl(C')$, and some $g\in G_\alpha$ lies in 
$\dcl(C')$. Thus, $ gE\in \dcl(C')$, and is the graph of an affine homomorphism,  denoted $g'$, from $\pi_1(G_\alpha)$ to $\pi_2(G_\alpha)$. (For ease of notation we do not
 distinguish between $gE$ and $\lceil gE \rceil$.) For $g_1,g_2\in G_\alpha$, write $g_1\sim g_2$ if $g_1,g_2$ determine the same affine homomorphism, that is, $g_1E=g_2E$. As in the proof of Claim 1(i) 
in the proof of Theorem~\ref{theorem1}, for any other affine homomorphism $h':\pi_1(G_\alpha)\to \pi_2(G_\alpha)$ with homogeneous component $\exp$, there is $d\in \mf{M}$ such that $g'(x)=h'(x)\exp(d)$ for all $x\in \pi_1(G_\alpha)$. Hence,

(2)  there is a $C'$-definable injection $j: (G_\alpha/\sim) \rightarrow K$.

If $h\in G_\alpha$ and $hE\in \acl(M,k)$ then, using elimination of finite imaginaries in ACF (applied to $K$), $j$ yields a definable map from a power of $k$ to 
$K$. Hence, by Lemma~\ref{finiteimage}, we have

(3) If $h\in G_\alpha$ and $hE\in \acl(M,k)$, then the image under $j$ of $\tp(hE/M,\alpha)$ is finite.

\begin{lemma}\label{rank0} Let $A=\acl(A)\supseteq M$ be any base set, let $\alpha \in G(k)$ and 
$h\in G_\alpha$, and assume that the coset $hE$ is $\acl(A\cup\{\alpha\})$-definable. Then $\alpha \in A$.
\end{lemma}

\begin{proof} Let $P$ be the set of realisations in $\Uu$ of $\tp(h/A)$, and put $\pi(P):=\{\pi(h): h\in P\}$. Thus $\pi(P)$ is the set of realisations of a type in $G(k)$ over $A$, so has Morley rank 0,1, or 2. It suffices to rule out the last two cases.

{\em Claim 1.} (i) If $\gamma\in k^*$ and $\beta\in G_a(k)$ is generic over $A\cup\{\gamma\}$ then $(\beta,\gamma)\not\in \pi(P)$.

(ii) If $\beta\in k$ and $\gamma\in G_m(k)$ is generic over $A\cup\{\beta\}$ then $(\beta,\gamma)\not\in \pi(P)$.

{\em Proof of Claim.} We prove (i) and omit the similar proof of  (ii). So suppose $\gamma\in k^*$, and $\beta_1,\beta_2\in G_a(k)$ are  generic and independent over $A\cup\{\gamma\}$. If (i) is false then there are $(b_1,c), (b_2,c)\in P$ with $(\beta_i,\gamma)=\pi(b_i,c)$  for $i=1,2$. By the assumption in the lemma, $(b_i,c)E$ is $\acl(A\cup \{\beta_i,\gamma\})$-definable for each $i$. Thus, working over $\acl(A,\beta_1,\beta_2,\gamma)$, there are definable affine homomorphisms $\beta_i+\mf{M} \to c(1+\mf{M})$ with homogeneous component $\exp$. Composing one with the inverse of the other, we have an $\acl(A,\beta_1,\beta_2,\gamma)$-definable bijection
between two cosets of $\mf{M}$ in $\calO$ which are generic over $A$. However, there is no such bijection: indeed, the product of these two cosets realises a unique type in $VF\times VF$ over $A \cup k$.

{\em Claim 2.} ${\rm RM}(\alpha/A)\leq 1$.

{\em Proof of Claim.} Suppose  ${\rm RM}(\alpha/A)=2$. Then $\pi(P)$ is a generic type of $G(k)$, which contradicts Claim 1.

{\em Claim 3.} Suppose $\gamma_1,\gamma_2\in \pi(P)$ and $\gamma_3=\gamma_1\gamma_2^{-1}$. Then ${\rm RM}(\gamma_3/A)\leq 1$.

{\em Proof of Claim.} Let $\gamma_i=\pi(g_i)$ with $g_i\in P$ for $i=1,2$. Put $g_3:=g_1g_2^{-1}$. By the assumption of the lemma, $g_iE$ and hence $g_i^{-1}E$ are $\acl(A,\gamma_i)$-definable for each $i$. Thus, $g_3E=g_1g_2^{-1}E$ is $\acl(A,\gamma_1,\gamma_2)$-definable. Hence $g_3E$ is $\acl(A,\gamma_3)$-definable by (3) above. Since this was the assumption (on $h,\alpha$) which yielded Claim 2, it follows by Claim 2 that $\tp(\gamma_3/A)$ has Morley rank at most 1.

\smallskip

Suppose now for a contradiction that ${\rm RM}(\alpha)={\rm RM}(\pi(P))=1$. 
Let $S$ be the Zilber stabiliser of $\pi(P)$ in the Morley rank 2 group $G(k)$. 
That is, if $p$ is the global non-forking extension (over $\Uu$) of 
the stationary type $\tp(\pi(h)/A)$, then $S=\{g\in G(k):g(p)=p\}$.  Then 
by $\omega$-stability of $k$,
  $S$ is a definable subgroup of $G(k)$.  
Also, by \cite[Lemme 2.3]{poizat}, ${\rm RM}(S)\leq {\rm RM}(p)$, and if there is equality then $S$ is connected and $p$ is a translate of the generic type of $S$. 
Let $g_1,g_2\models p$ with $g_1 \dnf_{\Uu} g_2$, and put $g_3:=g_1g_2^{-1}$. Then
$$1=\RM(g_2^{-1}/A,g_1)=\RM(g_1g_2^{-1}/A,g_1)=RM(g_3/A,g_1)\leq \RM(g_3/A)=1$$
(by Claim 3), so $g_3\dnf_{\Uu} g_1$, and as $g_3^{-1}g_1=g_2$, $g_3\in S$. 
Thus, $\RM(S)=\RM(p)=1$, and so  $p$  is the generic type  of a coset of $S$ in $G(k)$. 
By (1) above, and as $S$ is connected, $S=(\{0\}\times G_m)(k)$ or $S=(G_a\times\{1\})(k)$. Either of these gives a contradiction to Claim 1, so
 yields the lemma. 
\end{proof}

\medskip

{\em Proof of Theorem~\ref{theorem1}.} Let $A=\acl(A)\supseteq M$. Let $h\in G(\calO)$ such that ${\rm RM}(\pi(h)/A)=2$. We argue by contradiction, so suppose that $\lceil hE\rceil$ is coded in the geometric sorts. Now as in Claim 1 in our first proof of Theorem~\ref{theorem1},
there is $t\in K$ such that $hE\in \acl(A,\pi(h),t)$. By Lemma~\ref{rank0} applied over $A(t)$, it follows that $\pi(h)\in \acl(A,t)$. However, since ${\rm RM}(\pi(h)/A)=2$ and $t$ is a single field element, 
it follows easily from $C$-minimality arguments (see e.g. Section 3 of \cite{HK}) that
${\rm RM}(\pi(h)/\acl(A,t))\geq 1$. This  gives a contradiction. $\Box$

\begin{remark} \rm \label{final}
We have shown that the elements of the interpretable set $G(\calO)/E$ are not coded over any parameter set $A\supset M$, in the sense that for any such $A$, there is no $A$-definable injection from $G(\calO)/E$ to any product of geometric sorts. It follows that for any definable group $F>G(\calO)$ and any $b\in F$,  the elements of $bG/E$ are not coded over any set; for if the elements of $bG$ were coded over $A_b$, then the elements of $G/E$ would be coded over
$A_b \cup\{b\}$.

\end{remark}

\end{document}